\documentclass{amsart}

\usepackage[all]{xy}

\usepackage{amsmath,amssymb,amsthm,amsfonts,enumerate,array,color,lscape,fancyhdr,layout,pst-all}
\usepackage[english]{babel}
\usepackage[latin1]{inputenc}
\usepackage{young}

\newcounter{numberofremark}
\newcommand\nothing[1]{}
\usepackage[hyperindex,bookmarksnumbered,plainpages]{hyperref}

\newcommand{\dcl}{\DeclareMathOperator}
\dcl\C{\mathbb C}
\dcl\g{\mathfrak{g}}
\dcl\gl{\mathfrak{gl}}
\dcl\gr{gr}
\dcl\gts{\mathfrak{G}}
\dcl\K{\mathbb K}
\dcl\NN{\mathbb{N}}
\dcl\qdet{qdet}
\dcl\Sp{Specm}
\dcl\tr{\textup{tr}}
\dcl\wk{\mathbb{W}}
\dcl\Z{\mathbb{Z}}
\dcl\Y{Y}
\dcl\U{U}
\dcl\T{T}

\renewcommand\k{{\Bbbk}}

\newlength\yStones
\newlength\xStones
\newlength\xxStones

\makeatletter
\def\Stones{\pst@object{Stones}}
\def\Stones@i#1{%
  \pst@killglue%
  \begingroup%
  \use@par%
  \setlength\xxStones{\xStones}%
  \expandafter\Stones@ii#1,,\@nil
  \endgroup
  \global\addtolength\xStones{0.6cm}%
  \global\addtolength\yStones{-7.5mm}}%
\def\Stones@ii#1,#2,#3\@nil{%
  \rput(\xxStones,\yStones){%
    \psframebox[framesep=0]{%
      \parbox[c][6mm][c]{11mm}{\makebox[11mm]{$#1$}}}}%
  \addtolength\xxStones{1.2cm}%
  \ifx\relax#2\relax\else\Stones@ii#2,#3\@nil\fi}
\makeatother

\def\Stone#1{\fbox{\makebox[8mm]{\strut#1}}\kern2pt}

\newtheorem{theorem}{Theorem}[section]
\newtheorem{lemma}[theorem]{Lemma}
\newtheorem{corollary}[theorem]{Corollary}
\newtheorem{proposition}[theorem]{Proposition}
\newtheorem{example}[theorem]{Example}
\newtheorem{remark}[theorem]{Remark}
\newtheorem{notation}[theorem]{Notation}

\newtheorem{definition}[theorem]{Definition}

\begin{document}
\allowdisplaybreaks

\title{Gelfand-Tsetlin varieties for $\gl_n$}

\author{Germán Benitez Monsalve}
\address{\noindent Departamento de Matem\'atica, Instituto de Ciências Exatas, Universidade Federal do
Amazonas,  Manaus AM, Brazil}
\email{gabm03@gmail.com}


\begin{abstract}
S. Ovsienko proved that the Gelfand-Tsetlin variety for $\gl_n$ is equidimensional (i.e. all its irreducible components have the same dimension) with dimension equals $\frac{n(n-1)}{2}$. This result has important consequences in Repre\-sentation Theory of Algebras, implying, in particular, the equidimensionality of the nilfiber of the Kostant-Wallach map. In this paper we will present the generalization of this result and will address a version of Ovsienko's Theorem which includes the regular case. 

\end{abstract}
\maketitle

\tableofcontents    


\section{Introduction}

For the general linear Lie algebra $\gl_n$ over $\k$, consider the Gelfand-Tsetlin variety associated to the Gelfand-Tsetlin subalgebra $\Gamma$ for $\gl_n$. Ovsienko proved in \cite{Ovs} that this variety is equidimensional (that is, all its irreducible components have the same dimension) and this dimension equals $\frac{n(n-1)}{2}$. This result is called \textit{``Ovsienko's Theorem''} and, as a consequence of it, one obtains that the Gelfand-Tsetlin variety is a complete intersection. Independently, B. Kostant and N. Wallach in \cite{KW-1} and \cite{KW-2} proved that all regular components of the Gelfand-Tsetlin variety for $\gl_n$ have dimension $\frac{n(n-1)}{2}$.
 
A well-known Theorem of Kostant (Theorem ($0.13$) in \cite{K}) says that for a complex semisimple Lie algebra $\g$ the universal enveloping algebra $U(\g)$ is a free module over its center. When $\g=\gl_n$, the variety associated to the center of  the universal enveloping algebra $U(\gl_n)$ coincides with the variety of nilpotent matrices, which is  irreducible  of dimension $n^2-n$. A generalization of this result  was proved by V. Futorny and S. Ovsienko in \cite{FutOvs05} for a class of special filtered algebras. In particular, $U(\gl_n)$ is free as a left (or right) module over its Gelfand-Tsetlin subalgebra \cite{O}, the restricted Yangian $Y_p(\gl_n)$ of level $p$ for $gl_n$ is free over the center \cite{FutOvs05}. For the full Yangian $Y(\gl_n)$ this was shown previously by  A. Molev, M. Nazarov and G. Olshanski\u{i}  \cite{MNO}. In \cite{FMO} the authors showed that the finite $W$-algebras associated with $\gl_2$ (in particular, the restricted Yangian  $Y_p(\gl_2)$ of level $p$) is free over its Gelfand-Tsetlin subalgebra.



Gelfand-Tsetlin algebras appear in several contexts. Such algebras are related to important problems in representation theory of Lie algebras, for instances, E. B. Fomenko and A. S. Mischenko in \cite{FomMis} related such algebras in connection with the solutions of the Euler equation. Also, {\`E}. B. Vinberg in \cite{Vin} related the Gelfand-Tsetlin subalgebras in connection with subalgebras of ma\-xi\-mal Gelfand-Kirillov dimension of the universal enveloping algebra of a simple Lie algebra. B. Kostant and N. Wallach in {\cite{KW-1} and \cite{KW-2} used these subalgebras in connection with classical mechanics.

\vspace{0.5cm}

In this paper, we will study the Gelfand-Tsetlin variety for $\gl_n$, its equidimensio\-na\-lity and consequences of this. The paper is organized as follows: In section $\mathsection 2$ we introduce all necessary definitions, notations and results used throughout the paper, such as, equidimensionality of varieties, complete intersections and regular sequences. In section $\mathsection 3$ we recall the definition of Gelfand-Tsetlin subalgebras, Gelfand-Tsetlin varieties and Ovsienko's Theorem. In the same section we will prove a version of Ovsienko's Theorem (Theorem \eqref{fracagln}) using a different technique from \cite{Ovs} and as a consequence of this version we will show in Corollary \eqref{corVFraca} that, all the regular components of the Gelfand-Tsetlin variety for $\gl_n$ are equidimensional with dimension $\frac{n(n-1)}{2}$, moreover, all are isomorphic to the irreducible component $V\left(\left\{X_{ij}:1\leq i\leq j\leq n\right\}\right)\subset\k^{n^2}.$


It is  well-known that the fiber of zero $\Phi^{-1}(0)$  of the \textit{``Kostant-Wallach map''} \cite{KW-1}  
 $$\Phi:M_n(\C)\longrightarrow\C^{\frac{n(n+1)}{2}},$$
coincides with the Gelfand-Tsetlin variety of $\gl_n$ (Corollary \eqref{KWvsGTs}). M. Colarusso and S. Evens proved in \cite{ColEv}, that the fiber of zero $\Phi^{-1}_ 2(0)$ of the so-called \textit{``Partial Kostant-Wallach map''},   
 $$\Phi_2:M_n(\C)\longrightarrow\C^{n-1}\times\C^{n},$$
is equidimensional and in this case the dimension equals $n^2-2n+1$. Using different techniques, in Section $\mathsection 4$, we establish a generalization (Corollary \eqref{genColEv}) of this result, namely,
we define the \textit{``$k$-partial Kostant-Wallach map''}
 $$\Phi_k:M_n(\C)\longrightarrow\C^{n-k+1}\times\cdots\times\C^{n-1}\times\C^{n},$$
unifying the Kostant-Wallach map and the Partial Kostant-Wallach map and we prove that all its fibers $\Phi^{-1}_k(\alpha)$ for any $\alpha\in\C^{n-k+1}\times\cdots\times\C^{n-1}\times\C^{n}$ are equidimensional with dimension $n^2-(k+1)n+\frac{k(k+1)}{2}$.


Finally, in Section $\mathsection 5$ will be shown Ovsienko's Theorem for the cases $\gl_2,\gl_3$ and $\gl_4$ using a different technique from \cite{Ovs}.

\vspace{0,5cm}

\noindent{\bf Acknowledgements.} The author wishes to express his gratitude to Vyacheslav Futorny for suggesting the problem and for many stimulating conversations.


\section{Preliminaries}

Throughout the paper we fix an algebraically closed field $\k$ of characteristic zero.  

For a \textit{reduced} (without nilpotent elements) \textit{affine} $\k$-\textit{algebra} $\Lambda$ (that is, an associative and commutative finitely gene\-rated $\k$-algebra), we denote by $\Sp\Lambda$ the variety of all maximal ideals of $\Lambda$. If $\Lambda$ is a polynomial algebra in $n$ variables, then we identify $\Sp\Lambda$ with $\k^n$. For an ideal $I\subseteq\Lambda$, denote by $V(I)\subseteq\Sp\Lambda$, the set of all zeroes of $I$, called
\textit{variety}. If $I$ is generated by $g_1,g_2,\ldots,g_r$, then we write $I=(g_1,g_2,\ldots,g_r)$ and $V(I)=V(g_1,g_2,\ldots,g_r)$. A variety $V$ is an \textit{equidimensional variety} if all its irreducible components have the same dimension and we denote by $\dim V$ the dimension of $V$. 

We will denote the center of an associative filtrered algebra $\mathcal A$ by 
$Z(\mathcal A)$, the graded algebra of $\mathcal A$ by $\gr\mathcal(A)$ and the Universal enveloping algebra of a Lie algebra $\g$ by $U(\g)$.

Let $R$ be a commutative ring, $M$ be an $R-$module and $f:M\longrightarrow R$ be an $R-$linear map. The map
$$\begin{array}{rccl}
   & M^n                   & \longrightarrow & {\bigwedge} ^{n-1}M\\
   & (x_1,x_2,\ldots,x_n)  & \longmapsto     & \sum_{i=1}^n{(-1)^{i+1}f(x_i)x_1\wedge x_2\wedge\cdots\wedge\widehat{x}_i\wedge\cdots\wedge x_n}
                   \end{array}.$$
is an alternating $n-$linear map, where with $\widehat{x}_i$ we indicate that $x_i$ is to be omitted from the exterior product. By the universal property of the $n-$th exterior power there exists a unique $R$-linear map
$d_n^f:{\bigwedge} ^{n}M\longrightarrow {\bigwedge} ^{n-1}M$ with
$$\begin{array}{rccl}
d_n^f: & {\bigwedge} ^{n}M                   & \longrightarrow & {\bigwedge} ^{n-1}M\\
       & x_1\wedge x_2\wedge\ldots\wedge x_n & \longmapsto     & \sum_{i=1}^n{(-1)^{i+1}f(x_i)x_1\wedge x_2\wedge\cdots\wedge\widehat{x}_i\wedge\cdots\wedge x_n}
                   \end{array}.$$

\noindent Therefore, we have the chain complex of $R-$modules 
$$K_{\bullet}:=\xymatrix@C=14pt{\cdots\ar[r]&{\bigwedge} ^{n+1}M\ar[r]^-{d_{n+1}^f}&{\bigwedge} ^{n}M\ar[r]^-{d_{n}^f}&{\bigwedge} ^{n-1}M\ar[r]&\cdots\ar[r]^-{d_{3}^f}&{\bigwedge} ^{2}M\ar[r]^-{d_{2}^f}&M\ar[r]^-{f}&R\ar[r]&0}$$
called \textit{Koszul complex of} $f$ and denoted by $K_{\bullet}(f)$ (or $K_{\bullet}$).


Consider now a special case. Let $M$ be a free $R-$module with basis $e_1,e_2,\ldots,e_n$. Then, the $R$-linear map $f$ is determined uniquely by $x_i=f(e_i)\in R$ for any $i=1,2,\ldots,n$. Conversely, for a sequence
$\textbf{\textup{x}}=\{x_1,x_2,\ldots,x_n\}$ there exists an $R$-linear map $f$ over $M$ with $f(e_i)=x_i$. We denote the Koszul complex of $\textbf{\textup{x}}$ by
$$K_{\bullet}(\textbf{\textup{x}}):=K_{\bullet}(f).$$
                   

\noindent For an $R$-module $N$, the complex $K_{\bullet}(\textbf{\textup{x}},N):=K_{\bullet}(\textbf{\textup{x}})\otimes_R N$ is called \textit{Koszul complex of} $\textbf{\textup{x}}$ \textit{with coefficients in} $N$ and its homology is denoted by $H_n(\textbf{\textup{x}},N)$.
                   
Let $\textbf{\textup{x}}=\{x_1,x_2,\ldots,x_n\}$ be a sequence of elements in a ring $R$ and let $M$ be an $R$-module. We will say that the sequence is a \textit{complete intersection for} $M$ if  
$$H_{n}(\textbf{\textup{x}},M)=0, \ \ \forall n>0.$$
                   
A sequence $g_1,g_2,\ldots,g_t$ in a ring $R$ is called \emph{regular} if the image class of $g_i$ is not a zero divisor, and is not invertible in
$R/(g_1,g_2,\ldots,g_{i-1})$ for any $i=1,2,\ldots,t$.  


\begin{proposition}
\label{regperm}
Let $R$ be a noetherian ring and $g_1,g_2,\ldots,g_t$ a regular sequence of 
$R$. If $R$ is a graded ring and each $g_i$ is homogeneous of positive degree, then any permutation of $g_1,g_2,\ldots,g_t$ is regular in $R$.
\end{proposition}


\begin{proof} 
Theorem ($28$) in \cite{Mats70} (page $102$) or \cite{Mats89} (page $127$).

\end{proof}


\begin{corollary}
\label{regsubseq}
Let $R$ be a noetherian ring and $g_1,g_2,\ldots,g_t\in R$ a regular sequence of $R$. If $R$ is a graded ring and each $g_i$ is homogeneous of positive degree, then any subsequence of $g_1,g_2,\ldots,g_t$ is regular in $R$.
\end{corollary}


\begin{proof} 
By Proposition \eqref{regperm} and the definition of a regular sequence.

\end{proof}


\begin{proposition}
\label{defKoszul} 
Let $\textbf{\textup{g}}=\left\{g_1,g_2,\ldots,g_t\right\}$ be a regular sequence of a ring $R$ and let
$\textbf{\textup{e}}_1,\textbf{\textup{e}}_2,\ldots,\textbf{\textup{e}}_t$ be the standard basis of the free module $R^t$, then the sequence $\textbf{\textup{g}}$ is a complete intersection for $R$. Moreover, the sequence
$$\xymatrix{ \cdots\ar[r] & \bigwedge^{2}(R^t)\ar[r]^(.6){\partial} & R^t\ar[r]^{\eta} & R\ar[r]^(.25){\pi} & R/(g_1,g_2,\ldots,g_t)\ar[r] & 0 }$$
is exact, where $\pi$ is the projection, $\bigwedge^{i}$ is the $i-$th exterior power,  
$$\partial(\textbf{\textup{e}}_i\wedge \textbf{\textup{e}}_j)=g_i\textbf{\textup{e}}_j-g_j\textbf{\textup{e}}_i\ \ \ \ \ \ \ \mbox{and}\ \ \ \ \ \ \ \eta(f_1,f_2,\ldots,f_t)=\sum_{i=1}^t{f_ig_i}.$$
That sequence is called \textit{Koszul resolution}.
\end{proposition}


\begin{proof}
 Proposition ($5$) in \cite{Bourbaki} (page $157$).
 
\end{proof}


\begin{proposition}
\label{regvseq}
Let $R$ be an affine algebra of Krull dimension $n$, and 
$g_1,g_2,\ldots,g_t$ be a sequence of elements in $R$ with $0\leq t\leq n$. 

\begin{enumerate}[i.]
 \item 
If $R$ is graded and $g_1,g_2,\ldots,g_t$ are homogeneous, then
$g_1,g_2,\ldots,g_t$ is regular in $R$ if and only if the sequence
$g_1,g_2,\ldots,g_t$ is a complete intersection for $R$. 
 
 \item If $R$ is a Cohen-Macaulay algebra, then the sequence $g_1,g_2,\ldots,g_t$ is a complete intersection for $R$ if and only if the variety $V(g_1,g_2,\ldots,g_t)$ is equidimensional of dimension $n-t$.
\end{enumerate}
\end{proposition}


\begin{proof} 
Proposition ($2.1$) in \cite{FutOvs05}.

\end{proof}


\begin{proposition}
\label{projhyper}
Let $R=\k[X_1,X_2,\ldots,X_n]$ be a polynomial algebra, and let
$G_1,G_2,\ldots,G_t\in R$. The sequence 
$$X_1,X_2,\ldots,X_r,G_1,G_2,\ldots,G_t$$
is a complete intersection for $R$ if and only if the sequence
$g_1,g_2,\ldots,g_t$ is a complete intersection for
$\k[X_{r+1},X_{r+2},\ldots,X_{n}]$, where 
$$g_i(X_{r+1},X_{r+2},\ldots,X_{n})=G_i(0,0,\ldots,0,X_{r+1},X_{r+2},\ldots,X_{n}),\ \ \forall i=1,2,\ldots,t.$$
\end{proposition}


\begin{proof} 
Lemma ($2.2$) in \cite{FutOvs05}.

\end{proof}



\section{A version of the Ovsienko's Theorem}


From now on we fix some positive integer $n$. 


\subsection{Gelfand-Tsetlin subalgebra}
By $\gl_n$ we denote the general linear Lie algebra consisting of all $n\times n$ matrices over $\k$, and by $\{E_{ij}\mid 1\leq i,j \leq n\}$ the standard basis of $\gl_n$ of elementary matrices. For $m\leq n$, let $\gl_m$ be the Lie subalgebra of $\gl_n$ spanned by $\{E_{ij}\mid 1\leq i,j \leq m\}$ and denote by $Z_m:=Z(U(\gl_m))$ the center of $U(\gl_m)$. We have the chain of Lie subalgebras of $\gl_n$
$$\gl_1\subset\gl_2\subset\cdots\subset\gl_{n-1}\subset\gl_n$$
and the induced chain of subalgebras of $U(\gl_n)$
$$U(\gl_1)\subset U(\gl_2)\subset\cdots\subset U(\gl_{n-1})\subset U(\gl_n).$$


\begin{definition}
The \textbf{Gelfand-Tsetlin subalgebra $\Gamma$ of} $U(\gl_n)$ is defined to be the subalgebra $\Gamma$ of $U(\gl_n)$ generated by $\left\{Z_1,Z_2,\ldots,Z_n\right\}$.
\end{definition}


\begin{proposition}[\v Zelobenko, 1973]
For any $m\in\left\{1,2,\ldots,n\right\}$, the center $Z_m$ is a polynomial algebra in $m$ variables $\left\{\gamma_{mj}\ :\ j=1,2,\ldots,m\right\}$, with
 $$\gamma_{ij}=\sum_{t_1,t_2,\ldots,t_j\in \left\{1,2,\ldots,i\right\}}{E_{t_1t_2}E_{t_2t_3}\cdots E_{t_{j-1}t_j}E_{t_jt_1}}.$$
The subalgebra $\Gamma$ is a polynomial algebra in $\frac{n(n+1)}{2}$ variables $\left\{\gamma_{ij}\ :\ 1\leq j\leq i\leq n\right\}$.
\end{proposition}


\begin{proof}
 See \cite{Zelobenko} page $169$.

\end{proof}


\begin{remark}
\label{contatraco}
Consider the matrix 

$$E:=\left(\begin{matrix}
 E_{11} & E_{12} & \cdots & E_{1n}\\
 E_{21} & E_{22} & \cdots & E_{2n}\\
 \vdots & \vdots & \ddots & \vdots\\
 E_{n1} & E_{n2} & \cdots & E_{nn}
\end{matrix}\right)\in M_n\left(U(\gl_n)\right)$$
and denote by $E_i$ the $i\times i$ submatrix in the upper left corner of $E$ for $i\in\left\{1,2,\ldots,n\right\}$.
 
\noindent Therefore, for $1\leq j\leq i\leq n$ 
$$\gamma_{ij}=\tr(E_i^j)=\sum_{t_1=1}^{i}\sum_{t_2=1}^{i}\cdots\sum_{t_j=1}^{i}{E_{t_1t_2}E_{t_2t_3}\cdots E_{t_{j-1}t_{j}}E_{t_jt_1}},$$
where $\tr$ denotes the trace function.
\end{remark}


\begin{example}
The Gelfand-Tsetlin subalgebra $\Gamma$ for $U(\gl_3)$ is generated by 
\begin{align*} 
\gamma_{11}=&E_{11},\\
\gamma_{21}=&E_{11}+E_{22},\\
\gamma_{22}=&E_{11}^2+E_{12}E_{21}+E_{21}E_{12}+E_{22}^2,\\
\gamma_{31}=&E_{11}+E_{22}+E_{33},\\
\gamma_{32}=&E_{11}^2+E_{12}E_{21}+E_{13}E_{31}+E_{21}E_{12}+E_{22}
^2+E_{23}E_{32}+E_{31}E_{13}+E_{32}E_{23}+E_{33}^2,\\
\gamma_{33}=&E_{11}^3+E_{11}E_{12}E_{21}+E_{11}E_{13}E_{31}+E_{12}E_{21}E_{11}+E_{12}E_{22}E_{21}+E_{12}E_{23}E_{31}+\\
	    &+E_{13}E_{31}E_{11}+E_{13}E_{32}E_{21}+E_{13}E_{33}E_{31}+E_{21}E_{11}E_{12}+E_{21}E_{12}E_{22}+\\
	    &+E_{21}E_{13}E_{32}+E_{22}E_{21}E_{12}+E_{22}^3+E_{22}E_{23}E_{32}+E_{23}E_{31}E_{12}+ E_{23}E_{32}E_{22}+\\
	    &+E_{23}E_{33}E_{32}+E_{31}E_{11}E_{13}+E_{31}E_{12}E_{23}+E_{31}E_{13}E_{33}+E_{32}E_{21}E_{13}+\\
	    &+E_{32}E_{22}E_{23}+E_{32}E_{23}E_{33}+E_{33}E_{31}E_{13}+E_{33}E_{32}E_{23}+E_{33}^3.\\
\end{align*}

\end{example}


\begin{remark}
 For other generators of the Gelfand-Tsetlin subalgebra $\Gamma$ see \cite{Gelfandetal} or \cite{Molevbook} (pages $246-250$).
\end{remark}


\subsection{Gelfand-Tsetlin variety}

Clearly, the element $\gamma_{ij}$ can be viewed as a polynomial in noncommutative variables $E_{ij}$'s. But, by the Poincare-Birkhoff-Witt Theo\-rem the graded algebra $\gr(U(\gl_n))$ is a polynomial algebra in variables $\overline{E}_{ij}$'s 
 $$\gr(U(\gl(n)))\cong\k\left[\overline{E}_{ij} : i,j=1,2,\ldots ,n\right].$$  

\noindent Therefore, the elements $\overline{\gamma}_{ij}$ are polynomials in commutative variables $\overline{E}_{ij}$'s and considering the notation $X_{ij}:=\overline{E}_{ij}$ for $i,j=1,2,\ldots ,n$  
 $$\overline{\gamma}_{ij}=\sum_{t_1,t_2,\ldots,t_j\in \left\{1,2,\ldots,i\right\}}{X_{t_1t_2}X_{t_2t_3}\cdots X_{t_{j-1}t_j}X_{t_jt_1}}.$$


\begin{definition}
The \textbf{Gelfand-Tsetlin variety for} $\gl_n$ is the algebraic variety 
$$V\left(\left\{\overline{\gamma}_{ij} : i=1,2,\ldots ,n;\ j=1,2,\ldots ,i \right\}\right)\subset \k^{n^2}.$$
\end{definition}
 
 
\begin{theorem}[\textbf{Ovsienko's Theorem}, $2003$]
\label{tmaOvs}
\noindent

The Gelfand-Tsetlin variety for $\gl_n$ is equidimensional of dimension $\frac{n(n-1)}{2}$.
\end{theorem}
 
 
\begin{proof}
 See \cite{Ovs}.
 
\end{proof}


\subsection{Weak version of Gelfand-Tsetlin variety}


\begin{notation}
\label{notacPX}
\noindent

\begin{enumerate}[i.]
  \item For $t\in\NN$, we will write $d(t)=\frac{t(t+1)}{2}$ and $e(t)=\frac{(t+2)(t-1)}{2}$.
  \item Let $P\in\k[x_1,x_2,\dots,x_n]$ be a polynomial and 
$\textbf{X}=\{x_{i_1},x_{i_1},\dots,x_{i_r}\}$ be a set of variables. We will denote by $P^{\textbf{X}}$ the polynomial obtained from $P$ by substituting $x_{i_1}=x_{i_2}=\cdots=x_{i_r}=0$. 
  \item For $n\in\NN$, define the set of variables 
 $$I_{n}:=\left\{X_{ij}:n\geq i>j\geq 1\right\}\cup\left\{X_{in}:i=1,2,\ldots,n-1\right\}.$$
  \item For $t\in\NN$ with $t\leq n$, define the set $I_n^{(t)}:=\left\{X_{ij}\in I_n:i,j\neq t\right\}.$

\end{enumerate}
\end{notation}


\begin{theorem}
\label{fracagln}
\noindent

The variety
 $$V_n:=V\left(\sigma_{n2},\sigma_{n3},\sigma_{n4},\ldots,\sigma_{nn}\right)\subset\k^{n^2-n-d(n-2)}=\k^{e(n)}$$
is equidimensional of dimension $\dim\left(V_n\right)=e(n)-\left(n-1\right)=d(n-1),$ where
 $$\sigma_{nj}=\sum_{n>t_1>t_2>\cdots>t_{j-1}\geq 1}{X_{nt_1}X_{t_1t_2}\cdots X_{t_{j-2}t_{j-1}}X_{t_{j-1}n}};\ \ \ j=2,3,\ldots,n.$$
\end{theorem}


\begin{remark}
\noindent

\begin{enumerate}[a.]
  \item If we consider the variables matrix $X=\left(X_{ij}\right)_{i,j=1}^n$, then the set $I_n$ can be viewed as a matrix (removing some variables in $X$) formed by the entries of the part strictly lower triangular of $X$ together with the last column of $X$ (without $X_{nn}$) and $I_n^{(m)}$ is obtained from $I_n$ by removing the $m-$th row and the $m-$th column. For example,
 $$I_4=\left(\begin{matrix}
            
                   &        &        & X_{14} \\
            X_{21} &        &        & X_{24} \\
            X_{31} & X_{32} &        & X_{34} \\
            X_{41} & X_{42} & X_{43} &        
           \end{matrix}\right)\ \ \ \ \ \ \mbox{and}\ \ \ \ \ \ I_5^{(1)}=\left(\begin{matrix}
            
                   &        &        & X_{25} \\
            X_{32} &        &        & X_{35} \\
            X_{42} & X_{43} &        & X_{45} \\
            X_{52} & X_{53} & X_{54} & 
            
           \end{matrix}\right).$$
  \item $\sigma_{nj}\in \k[I_n]\ ,\ \ \ \forall j=2,3,\ldots,n.$ 
\end{enumerate}
 
\end{remark}

\vspace{0.5cm}

The motivation for studying the equidimensionality of the variety $V_n$ is due to our belief that the Theorem \eqref{fracagln} implies Ovsienko's Theorem. For instance, $V_2$ is exactly the Gelfand-Tsetlin variety for $\gl_2$ and the Gelfand-Tsetlin variety for $\gl_3$ is the union of $V_3$ and another subvariety which isomorphic to $V_3$. Due to the fact that $V_n$ is a subvariety of the Gelfand-Tsetlin variety, we will call $V_n$  {\bf Weak Version of the Gelfand-Tsetlin variety}.


\begin{proof}
By induction on $n$. Clearly, 
 $$V_2=V(\sigma_{22})=V(X_{21}X_{12})=V(X_{21})\cup V(X_{12})\subset\k^2$$
is equidimensional of $\dim(V(X_{21}))=1=\dim(V(X_{12})).$

Suppose that 
 $$V_{n-1}=V\left(\sigma_{n-1,2},\sigma_{n-1,3},\ldots,\sigma_{n-1,n-1}\right)\subset\k^{e(n-1)}$$
is equidimensional of dimension $\dim\left(V_{n-1}\right)=d(n-2)$, where
 $$\sigma_{n-1,j}=\sum_{n-1>t_1>t_2>\cdots>t_{j-1}\geq 1}{X_{n-1,t_1}X_{t_1t_2}\cdots X_{t_{j-2}t_{j-1}}X_{t_{j-1},n-1}}$$
with $j=2,3,\ldots,n-1$. And we prove that
 $$V_n=V\left(\sigma_{n2},\sigma_{n3},\sigma_{n4},\ldots,\sigma_{nn}\right)\subset\k^{e(n)}$$
is equidimensional of dimension $\dim\left(V_n\right)=e(n)-\left(n-1\right)=d(n-1)$, where
 $$\sigma_{nj}=\sum_{n>t_1>t_2>\cdots>t_{j-1}\geq 1}{X_{nt_1}X_{t_1t_2}\cdots X_{t_{j-2}t_{j-1}}X_{t_{j-1}n}};\ \ \ j=2,3,\ldots,n.$$
Since $\sigma_{nn}=X_{n,n-1}X_{n-1,n-2}\cdots X_{32}X_{21}X_{1n},$ we obtain 
 $$V_n=V\left(\sigma_{n2},\sigma_{n3},\ldots,\sigma_{nn-1},X_{1n}\right)\cup\bigcup_{t=2}^{n}V\left(\sigma_{n2},\sigma_{n3},\ldots,\sigma_{nn-1},X_{tt-1}\right).$$

 For 
$V\left(\sigma_{n2},\sigma_{n3},\ldots,\sigma_{nn-1},X_{1n}\right)\subset\k^{e(n)}$, by Corollary \eqref{regsubseq} and Proposition \eqref{regvseq} it is enough to prove that\footnote{The variables $X_{i1}$ ($i=2,3,\ldots,n$) do not appear in the polynomials $\sigma_{nj}^{\textbf{X}}$ for any $j=2,3,\ldots,n$.}
 $$V_n^{(1)}:=V\left(\sigma_{n2},\sigma_{n3},\ldots,\sigma_{nn-1},X_{1n},X_{21},X_{31},X_{41},\ldots,,X_{n1}\right)\subset\k^{e(n)}$$ 
is equidimensional of dimension $\dim\left(V_n^{(1)}\right)=e(n)-(2n-2)=d(n-2)$. By Propositions \eqref{regvseq} and \eqref{projhyper}, this is equivalent to proving that the variety
 $$V_n^{(2)}:=V\left(\sigma_{n2}^{\textbf{X}},\sigma_{n3}^{\textbf{X}},\ldots,\sigma_{nn-1}^{\textbf{X}}\right)\subset\k^{e(n)-n}=\k^{e(n-1)}$$ 
is equidimensional of dimension $\dim\left(V_n^{(2)}\right)=e(n-1)-(n-2)=d(n-2)$, where 
 $$\textbf{X}=\{X_{1n},X_{21},X_{31},X_{41},\ldots,X_{n1}\}.$$

\noindent Note that \footnote{$\ \sigma_{nj}^{\textbf{X}}\neq 0\ ,\ \ \forall j=2,3,\ldots,n-1.$}
 $$\sigma_{nj}^{\textbf{X}}=\sum_{n>t_1>t_2>\cdots>t_{j-1}> 1}{X_{nt_1}X_{t_1t_2}\cdots X_{t_{j-2}t_{j-1}}X_{t_{j-1}n}}\in\k[\, I_n^{(1)}\ ];\ \ j=2,3,\ldots,n-1$$ 
because $t_s\neq 1,\ \ \forall s=1,2,\ldots,j-1$ and considering the $\k$-algebra isomorphism 

$$\begin{array}{rccl}
  \varphi: & \k[I_{n-1}] & \longrightarrow & \k\left[I_n^{(1)}\right] \\
           & X_{ij}  & \longmapsto     & \varphi(X_{ij}):=X_{i+1,j+1}
                   \end{array},$$
we have that, for any $j=2,3,\ldots,n-1$
 \begin{align*}
\varphi(\sigma_{n-1,j}) & =\sum_{n-1>t_1>t_2>\cdots>t_{j-1}\geq 1}{X_{n,t_1+1}X_{t_1+1,t_2+1}\cdots X_{t_{j-2}+1,t_{j-1}+1}X_{t_{j-1}+1,n}}\\
						& =\sum_{n>t_1>t_2>\cdots>t_{j-1}>1}{X_{nt_1}X_{t_1,t_2}\cdots X_{t_{j-2},t_{j-1}}X_{t_{j-1},n}}\\
						& =\sigma_{nj}^{\textbf{X}}.
 \end{align*}

\noindent Therefore $V_n^{(2)}\cong V_{n-1}$, which by the inductive hypothesis is equidimensional of dimension
 $$\dim\left(V_n^{(2)}\right)=\dim\left(V_{n-1}\right)=d(n-2).$$ 

For $V\left(\sigma_{n2},\sigma_{n3},\ldots,\sigma_{nn-1},X_{tt-1}\right)\subset\k^{e(n)}$ 
with $t\in\left\{2,3,\cdots,n-1\right\}$. Applying the Corollary \eqref{regsubseq} and Proposition \eqref{regvseq} it is enough to prove that the algebraic variety $V_n^{(1)}\subset\k^{e(n)}$ defined by 
 $$V\left(\sigma_{n2},\sigma_{n3},\ldots,\sigma_{nn-1},X_{tt-1},X_{t1},X_{t2},\ldots,,X_{t,t-2},X_{tn},X_{t+1,t},X_{t+2,t},\ldots,,X_{nt}\right)$$
is equidimensional of dimension $\dim\left(V_n^{(1)}\right)=e(n)-(2n-2)=d(n-2)$. By Propositions \eqref{regvseq} and \eqref{projhyper} it is equivalent to prove that 
 $$V_n^{(2)}:=V\left(\sigma_{n2}^{\textbf{X}},\sigma_{n3}^{\textbf{X}},\ldots,\sigma_{nn-1}^{\textbf{X}}\right)\subset\k^{e(n)-n}=\k^{e(n-1)}$$
is equidimensional of dimension $\dim\left(V_n^{(2)}\right)=e(n-1)-(n-2)=d(n-2)$, where 
 $$\textbf{X}=\{X_{tt-1},X_{t1},X_{t2},\ldots,,X_{t,t-2},X_{tn},X_{t+1,t},X_{t+2,t},\ldots,,X_{nt}\}.$$
We note that \footnote{ $\sigma_{nj}^{\textbf{X}}\neq 0$, because $\sigma_{nj}^{\textbf{X}}$ has the monomial $X_{n,j-1}X_{j-1,j-2}\cdots X_{32}X_{21}X_{1n}$ when $j\leq t$ and has $X_{n,n-1}X_{n-1,n-2}\cdots X_{n-(j-t+1),n-(j-t)}X_{n-(j-t),t-1}X_{t-1,t-2}X_{t-2,t-3}\cdots X_{32}X_{21}X_{1n}$ if $j>t$} 
 $$\sigma_{nj}^{\textbf{X}}=\sum_{\stackrel{n>t_1>t_2>\cdots>t_{j-1}\geq 1}{t_s\neq t}}{X_{nt_1}X_{t_1t_2}\cdots X_{t_{j-2}t_{j-1}}X_{t_{j-1}n}};\ \ \ j=2,3,\ldots,n-1,$$ 
because $t_s\neq t,\ \forall s=1,2,\ldots,j-1$ and considering the $\k$-algebra isomorphism

$$\begin{array}{rccl}
  \varphi: & \k[I_{n-1}] & \longrightarrow & \k\left[I_n^{(t)}\right] \\
           & X_{ij}  & \longmapsto     & \varphi(X_{ij}):=
           \left \{\begin{matrix}
	    X_{ij}\ \ \ \ \ \ \ ; & \mbox{if}\ i,j<t\\
	    X_{i+1,j}\ \ \ ;      & \ \ \, \mbox{if}\ j<t\leq i\\
	    X_{i,j+1}\ \ \ ;      & \ \ \, \mbox{if}\ i<t\leq j\\
	    X_{i+1,j+1};          & \mbox{if}\ t\leq i,j\\ 
           \end{matrix}\right.
   \end{array},$$
we have that, for any $j=2,3,\ldots,n-1$
 \begin{align*}
\varphi(\sigma_{n-1,j}) & =\sum_{\stackrel{n>t_1>t_2>\cdots>t_{j-1}\geq 1}{t_s\neq t}}{X_{nt_1}X_{t_1,t_2}\cdots X_{t_{j-2},t_{j-1}}X_{t_{j-1},n}}=\sigma_{nj}^{\textbf{X}}.
 \end{align*}
Therefore $V_n^{(2)}\cong V_{n-1}$, which by the inductive hypothesis is equidimensional of dimension 
$$\dim\left(V_n^{(2)}\right)=\dim\left(V_{n-1}\right)=d(n-2).$$
 
For $V\left(\sigma_{n2},\sigma_{n3},\dots,\sigma_{nn-1},X_{nn-1}\right)\subset\k^{e(n)}$. Applying the Corollary \eqref{regsubseq} and Proposition \eqref{regvseq} it is enough to prove that the variety $V_n^{(1)}\subset\k^{e(n)}$ defined by \footnote{Note that, here we can not use the same technique of the  previous cases, because by increasing the $n-$th column or the $n-$th row, we obtain $\sigma_{nj}^{\textbf{X}}=0$ for any $j=2,3,\ldots,n-1$ and therefore the sequence 
$\sigma_{n2}^{\textbf{X}},\sigma_{n3}^{\textbf{X}},\sigma_{n4}^{\textbf{X}},\ldots,\sigma_{n,n-1}^{\textbf{X}}$ is not a regular sequence. In this case, we are only going to increase the $(n-1)-$th row of the matrix $I_n$.}
 $$V\left(\sigma_{n2},\sigma_{n3},\ldots,\sigma_{nn-1},X_{nn-1},X_{n-1,1},X_{n-1,2},\ldots,X_{n-1,n-2},X_{n-1,n}\right)$$
is equidimensional of dimension $\dim\left(V_n^{(1)}\right)=e(n)-(2n-2)=d(n-2)$, which by Propositions \eqref{regvseq} and \eqref{projhyper}, it is equivalent to proving that 
 $$V_n^{(2)}:=V\left(\sigma_{n2}^{\textbf{X}},\sigma_{n3}^{\textbf{X}},\ldots,\sigma_{nn-1}^{\textbf{X}}\right)\subset\k^{e(n)-n}=\k^{e(n-1)},$$
where $\textbf{X}=\{X_{n,n-1},X_{n-1,1},X_{n-1,2},\ldots,X_{n-1,n-2},X_{n-1,n}\}$, is equidimensional with dimension $\dim\left(V_n^{(2)}\right)=e(n-1)-(n-2)=d(n-2)$. Firstly, for any $j=2,3,\ldots,n-1$ 
$$\sigma_{nj}^{\textbf{X}}=\sum_{n-1>t_1>t_2>\cdots>t_{j-1}\geq 1}{X_{nt_1}X_{t_1t_2}\cdots X_{t_{j-2}t_{j-1}}X_{t_{j-1}n}}.$$
Now, considering the $\k$-algebra isomorphism 

 $$\begin{array}{rccl}
\varphi: & \k[I_{n-1}] & \longrightarrow & \k\left[I_n^{(n-1)}\right] \\
         & X_{ij}  & \longmapsto     & \varphi(X_{ij}):=
    \left \{\begin{matrix}
	    X_{ij};   & \ \ \, \mbox{if}\ i,j\neq n-1\\
	    X_{nj};   & \mbox{if}\ i=n-1\\
	    X_{in};   & \, \mbox{if}\ j=n-1\\
            \end{matrix}\right.
   \end{array},$$
it follows that, for any $j=2,3,\ldots,n-1$
 \begin{align*}
\varphi(\sigma_{n-1,j}) & =\sum_{n-1>t_1>t_2>\cdots>t_{j-1}\geq 1}{X_{nt_1}X_{t_1,t_2}\cdots X_{t_{j-2},t_{j-1}}X_{t_{j-1},n}}=\sigma_{nj}^{\textbf{X}}
 \end{align*}
and therefore $V_n^{(2)}\cong V_{n-1}$. Hence by the inductive hypothesis we have that $V_n^{(2)}$ is equidimensional of dimension $\dim\left(V_n^{(2)}\right)=\dim\left(V_{n-1}\right)=d(n-2)$.
 
Finally, we conclude that the variety $V_n$ is equidimensional with dimension
  $$\dim\left(V_n\right)=d(n-1).$$

\end{proof}


\subsection{Regular components of the Gelfand-Tsetlin variety}

We will say that a subvariety $V$ of the Gelfand-Tsetlin variety of $\gl_n$ is {\it regular} if $V$ has the form $V\left(f_1,f_2,\ldots,f_t\right)\subset\k^{n^2}$, where each polynomial is a variable, that is, for each $i$
$$f_i=X_{rs},\ \ \mbox{for some}\ \ 1\leq r,s\leq n.$$
Without loss of generality, we will assume that $f_i\neq f_j,\ \forall i\neq j$.


\begin{remark}
\label{isocompreg}
Let $V_1=V(X_{i_1},X_{i_2},\ldots,X_{i_r})$ and $V_2=V(X_{j_1},X_{j_2},\ldots,X_{j_r})$ be two regular subvarieties of a variety $V\subseteq\k^m$, so there exists a permutation $\sigma\in S_m$ such that 
$$\sigma(i_t)=j_t,\ \forall t=1,2,\ldots,r.$$
Therefore, we have the $\k$-algebra isomorphism
$$\begin{array}{rccl}
  \varphi_{\sigma}: & \k[X_1,X_2,\ldots,X_m] & \longrightarrow & \k[X_1,X_2,\ldots,X_m]  \\
           & X_i  & \longmapsto     & \varphi_{\sigma}(X_i):=X_{\sigma(i)}
    \end{array}.$$
Since\footnote{Given a subset $V\subseteq\k^n$, we denote by $I(V)$ the ideal of all polynomials vanishing on $V$.} $\varphi_{\sigma}\left(I\left(V_2\right)\right)=I\left(V_1\right)$, follows that $I\left(V_2\right)$ is a minimal prime ideal if and only if $I\left(V_1\right)$ is a minimal ideal. Equivalently, $V_2=V(X_{j_1},X_{j_2},\ldots,X_{j_r})$ is an irreducible component of $V$ if and only if $V_1=V(X_{i_1},X_{i_2},\ldots,X_{i_r})$ is an irreducible component of $V$.

\end{remark}

\vspace{0.5cm}

As a corollary from Theorem \eqref{fracagln}, all the regular components of the Gelfand-Tsetlin variety for $\gl_n$ are equidimensionals with dimension $\frac{n(n-1)}{2}$.

\vspace{0.5cm}

\begin{corollary}
\label{corVFraca}
All the regular components of the Gelfand-Tsetlin variety are isomorphic. In particular the irreducible components and 
$$V_{\leq}:=V\left(\left\{X_{ij}:1\leq i\leq j\leq n\right\}\right).$$
are isomorphic.
\end{corollary}


\begin{proof}

Let $V\left(f_1,f_2,\ldots,f_{m}\right)\subset\k^{n^2}$ be a regular component of the Gelfand-Tsetlin variety. Clearly, for any $i\in\left\{1,2,\ldots,n\right\}$, there exists $t\in\left\{1,2,\ldots,m\right\}$ such that $f_t=X_{ii}$, otherwise there exists $A\in V\left(f_1,f_2,\ldots,f_{m}\right)\subset\k^{n^2}$ satisfying $\overline{\gamma}_{i1}(A)=A_{ii}\neq 0$, for instance, the elementary matrix $E_{ii}$. 

Also, for any $i,j\in\left\{1,2,\ldots,n\right\}$ with $i\neq j$, there exists $t\in\left\{1,2,\ldots,m\right\}$ such that $f_t=X_{ij}$ or $f_t=X_{ji}$, for otherwise, if $i>j$ there exists $A\in V\left(f_1,f_2,\ldots,f_{m}\right)$ satisfying
$$\overline{\gamma}_{i2}(A)=\sum_{t=1}^{i-1}{A_{it}A_{ti}}=\sum_{t=1}^{j-1}{A_{it}A_{ti}}+\underbrace{A_{ij}A_{ji}}_{\neq 0}+\sum_{t=j+1}^{i-1}{A_{it}A_{ti}}\neq 0,$$
for instance, $A=E_{ij}+E_{ji}$ and similarly for $i<j$. Therefore, $m\geq\frac{n(n+1)}{2}.$

\noindent If $m>\dfrac{n(n+1)}{2}$, then there are $i,j\in\left\{1,2,\ldots,n\right\}$ with $i>j$ such that 
$$X_{ij},X_{ji}\in\left\{f_1,f_2,\ldots,f_{m}\right\}.$$
Consider the set $S\subset \left\{X_{ij}:i,j=1,2,\ldots,n\right\}$ with the following conditions:
\begin{enumerate}
	\item $S$ has $m$ elements.
	\item $\left\{X_{ij}:1\leq i\leq j\leq n\right\}\subseteq S$.	
\end{enumerate}
Hence $V(S)$ is a subvariety of the Gelfand-Tsetlin variety and by the Remark \eqref{isocompreg}
 $$V\left(f_1,f_2,\ldots,f_{m}\right)\cong V(S).$$
Note that $V(S)$ is irreducible, but it is not a component, because $V(S)\subsetneq V_{\leq}$ and this implies that $V(S)$ is not a maximal subva\-rie\-ty, which contradicts our assumption that  $V\left(f_1,f_2,\ldots,f_{m}\right)$ 
is a regular component.

\noindent Hence, we can conclude that if
$V\left(f_1,f_2,\ldots,f_{m}\right)$ is a regular component, then
$$m=\frac{n(n+1)}{2}$$
and therefore, by Remark we have \eqref{isocompreg} $V\left(f_1,f_2,\ldots,f_{d(n)}\right)\cong V_{\leq}.$

\end{proof}


\section{Equidimensionality for the fibers of the Kostant-Wallach map and of its partial maps}
\label{sec:koswall}


\subsection{PBW algebras and Special filtered algebras}

A \textit{filtered algebra} is an associative algebra $U$ over $\k$, endowed with an increasing filtration $\{U_i\}_{i\geq 0}$, where 
$$U_0=\k,\ \ \ U_iU_j\subseteq U_{i+j}\ \ \ \mbox{and}\ \ \ U=\bigcup_{i=0}^{\infty}U_i.$$

\noindent For $u\in U_i\setminus U_{i-1}$ set $\deg(u)=i$ (by convention $U_{-1}=\{0\}$).

The \textit{associated graded algebra} of $U$ is  
$$\overline{U}=\gr(U)=\bigoplus_{i=0}^{\infty}U_i/U_{i-1}.$$ 

\noindent For $u\in U$ denote by $\overline{u}$ its image in $\overline{U}$. 

An algebra $U$ is called a \textit{PBW algebra} if any element of $U$ can be written uniquely as a linear combinations of ordered monomials in some fixed generators of $U$ and if $U$ is a PBW algebra such that $\overline{U}$ is a polynomial algebra then $U$ will be called \textit{special filtered} or \textit{special PBW}.


\begin{remark}
By the Poincare-Birkhoff-Witt theorem the universal envelo\-ping algebra $U(\g)$ of any finite-dimensional Lie algebra $\g$ is a special filtered algebra. 
\end{remark}


Introduce the mapping $(\, , )_U:U^t\times U^t\longrightarrow U$, such that, for $\textbf{u}=(u_1,u_2,\ldots,u_t)$ and $\textbf{v}=(v_1,v_2,\ldots,v_t)$
$$(\textbf{u},\textbf{v})_U=\sum_{i=1}^t{u_iv_i}.$$ 


\begin{lemma}
\label{PBW}
Let $U$ be a special filtered algebra, $g_1,g_2,\ldots,g_t\in U$ be mutually commuting elements such that
$\overline{g}_1,\overline{g}_2,\ldots,\overline{g}_t$ is regular in $\gr(U)=\overline{U}$ and $I\subseteq U$ be a left ideal of $U$ generated by $g_1,g_2,\ldots,g_t$. Then, any $f\in I$ can be written in the form 
$$f=\sum_{j=1}^{t}f_jg_j$$
for some $f_j\in U$, $j=1,2,\ldots,t$, such that 
$$\deg(f)=\max_{1\leq j\leq t}\deg(f_jg_j).$$
\end{lemma}


\begin{proof}
Suppose that $f\in I$ and consider $f=f_1g_1+f_2g_2+\cdots +f_tg_t$ with the minimal possible 
$$d=\max_{1\leq i\leq t}deg(f_ig_i).$$
We may assume that, $d=\deg(f_ig_i)$ if and only if $i=1,2,\ldots,r$ for some $r\leq t$. Consider $\textbf{f}=(f_1,f_2,\ldots,f_t)$,
$\textbf{g}=(g_1,g_2,\ldots,g_t)$ and for a vector
$\textbf{s}=(s_1,s_2,\ldots,s_t)\in U^t$ denote
$\overline{\textbf{s}}=(\overline{s}_1,\overline{s}_2,\ldots,\overline{s}_r,0,0,\ldots,0)\in \overline{U}^t$.

If $d>\deg(f)$ then from condition $(\textbf{f},\textbf{g})_U=f$ follows $(\overline{\textbf{f}},\overline{\textbf{g}})_{\overline{U}}=0$ and by the Koszul resolution in the proposition \eqref{defKoszul} we have 
$\eta(\overline{f}_1,\overline{f}_2,\ldots,\overline{f}_r)=0$, that is,
$(\overline{f}_1,\overline{f}_2,\ldots,\overline{f}_r)\in\ker(\eta)$. Since $U$ is a special filtered algebra, the sequence
$\overline{g}_1,\overline{g}_2,\ldots,\overline{g}_r$ is regular in
$\overline{U}$ by corollary \eqref{regsubseq}. Again, by the Koszul resolution in the proposition \eqref{defKoszul} we have $(\overline{f}_1,\overline{f}_2,\ldots,\overline{f}_r)\in Im(\partial)$, therefore, there exists $g_{ij}\in U,\ 1\leq i<j\leq r$ such that 
$$\overline{\textbf{f}}=\sum_{1\leq i<j\leq r}{\overline{g}_{ij}(\overline{g}_i\textbf{e}_j-\overline{g}_j\textbf{e}_i)},$$ where $\textbf{e}_1,\textbf{e}_2,\ldots,\textbf{e}_r$ is a standard basis of $\overline{U}^r$.

Now, we consider 
$$\textbf{k}=\sum_{1\leq i<j\leq r}{g_{ij}(g_i\textbf{e}_j-g_j\textbf{e}_i)}\in U^t$$
then $\overline{\textbf{f}}=\overline{\textbf{k}}$, hence for $\textbf{h}=(h_1,h_2,\ldots,h_t)=\textbf{f}-\textbf{k}$ holds $\deg(h_i)\leq\deg(f_i)$ for all $i=1,2,\ldots,t$ and $\deg(h_i)<\deg(f_i)$ for all $i=1,2,\ldots,r$. But, note that $(g_i\textbf{e}_j-g_j\textbf{e}_i,\textbf{g})_U=g_ig_j-g_jg_i=0$, hence
$(\textbf{k},\textbf{g})_U=0$. Then
$$\sum_{i=1}^t{h_ig_i}=(\textbf{h},\textbf{g})_U=(\textbf{f}-\textbf{k},\textbf{g})_U=(\textbf{f},\textbf{g})_U=f.$$
Since 
$$\max_{1\leq i\leq t}\deg(h_ig_i)<\max_{1\leq i\leq t}\deg(f_ig_i)$$
contradicts the minimality of $d$, we therefore conclude that $d=\deg(f)$.

\end{proof}


Before starting the following proposition, we note that if $U$ is a special filtered algebra, the grading on $\overline{U}=\gr(U)$ does not coincide in general with
its standard grading as a polynomial algebra. In particular, if $d_1,d_2,\ldots, d_n$ are positive integers then $\Lambda=\overline{\Lambda}=\k[X_1,X_2,\ldots,X_n]$, endowed with a grading $\deg(X_i)=d_i$, is a special filtered algebra with respect to the corresponding
filtration. When all $d_i=1$ we get a standard grading on $\Lambda$.


\begin{proposition}
\label{regseq}
Let $\Lambda=\k[X_1,X_2,\ldots,X_n]$ and $g_1,g_2,\ldots,g_t\in\Lambda$ be such that the sequence $\overline{g}_1,\overline{g}_2,\ldots,\overline{g}_t$ is
regular in $\Lambda=\gr(\Lambda)$. Then, the sequence $g_1,g_2,\ldots,g_t$ is regular in $\Lambda$.
\end{proposition}


\begin{proof} 
Suppose that the sequence $g_1,g_2,\ldots,g_t$ is not regular in $\Lambda$, hence for some $i\in\left\{1,2,\ldots,t\right\}$ such that $g_i$ is a zero divisor or invertible in $\Lambda/(g_1,g_2,\ldots,g_{i-1})$.

If $g_i$ is a zero divisor in $\Lambda/(g_1,g_2,\ldots,g_{i-1})$, then there exists $f\in\Lambda$ such that $f\notin(g_1,g_2,\ldots,g_{i-1})$ and  $fg_i\in(g_1,g_2,\ldots,g_{i-1})$, which implies that
$$fg_i=\sum_{j=1}^{i-1}f_jg_j, \ \ \mbox{for some}\ \ f_j\in\Lambda.$$
We consider $f$ with the minimal degree possible, then by lemma \eqref{PBW} 
$$\deg(fg_i)=\max_{1\leq j\leq i-1}\deg(f_jg_j)$$
hence $\overline{f}\overline{g}_i$ is zero in
$\Lambda/(\overline{g}_1,\overline{g}_2,\ldots,\overline{g}_{i-1})$. Since the sequence $\overline{g}_1,\overline{g}_2,\ldots,\overline{g}_{t}$ is regular, then $\overline{g}_1,\overline{g}_2,\ldots,\overline{g}_{i-1}$ is also regular. Therefore $\overline{f}$ is zero in
$\Lambda/(\overline{g}_1,\overline{g}_2,\ldots,\overline{g}_{i-1})$, hence $$\overline{f}=\sum_{j=1}^{i-1}\overline{h}_j\overline{g}_j,\ \ \mbox{for some}\ \ h_j\in\Lambda.$$
It follows that 
$$f'=f-\sum_{j=1}^{i-1}h_jg_j$$ 
has a smaller degree than $f$. On the other hand, since
$f\notin(g_1,g_2,\ldots,g_{i-1})$, we obtain $f'\notin(g_1,g_2,\ldots,g_{i-1})$ and
$f'g_i\in (g_1,g_2,\ldots,g_{i-1})$. This contradicts the mi\-ni\-ma\-li\-ty of the degree of $f$, therefore $g_i$ is not a zero divisor in $\Lambda/(g_1,g_2,\ldots,g_{i-1})$. The case when the image $g_i$ in $\Lambda/(g_1,g_2,\ldots,g_{i-1})$ is inver\-tible is treated analogously.

\end{proof}



\subsection{Kostant-Wallach map and its partial map}

Let $n$ be a positive integer and $X\in M_n(\C)$ be a matrix. For $i=1,2,\ldots,n$ let $\chi_i(X)\in\C^i$ be the vector $\chi_{i}(X):=\left(\chi_{i1}(X),\chi_{i2}(X),\ldots,\chi_{ii}(X)\right)$, where $\chi_{ij}(X)$ is the coefficient of the powers $t^{i-j}$ of the the characteristic polynomial $X_i$ (for $j=1,2,\ldots,i$) and $X_i$ denote the $i\times i$ submatrix in the upper left corner of $X$.


\begin{definition}
\noindent
\begin{enumerate}
 \item The \textbf{Kostant-Wallach map} is the morphism given by 
$$\begin{array}{rccl}
 \Phi: & M_n(\mathbb{C}) & \longrightarrow & \mathbb{C}^{\frac{n(n+1)}{2}}\\
       & X               & \longmapsto     & \Phi(X):=\left(\chi_1(X),\chi_2(X),\ldots,\chi_n(X)\right).
 \end{array}.$$

\item For $k=1,2,\dots,n$, the \textbf{$k$-partial Kostant-Wallach map} is the morphism $\Phi_k:=\pi_k\circ\Phi$, with
$\pi_k:\C\times\C^2\times\cdots\C^n\longrightarrow\C^{n-k+1}\times\cdots\times\C^{n-1}\times\C^{n}$ the projection on the last $k$ factors, i.e.
$$\begin{array}{rccl}
   \Phi_k: & M_n(\C) & \longrightarrow & \C^{n-k+1}\times\cdots\C^{n-1}\times\C^{n}\\
           & X       & \longmapsto     & \Phi_k(X):=\left(\chi_{n-k+1}(X),\ldots,\chi_{n-1}(X),\chi_{n}(X)\right).
  \end{array}$$

 \end{enumerate}
 
\end{definition}

\begin{remark}
 Clearly $\Phi_n=\Phi$.
\end{remark}

 For any $\alpha\in\C^{n-k+1}\times\cdots\C^{n-1}\times\C^{n}$, let $\Phi_k^{-1}(\alpha)$
be the {\bf fiber} of $\Phi_k$.


\begin{remark}
\label{remarkEig}
For any $X\in M_n(\C)$ and $1\le i\le n$, let 
$${\bf E}_X(i):=\{\lambda_{i1}(X),\lambda_{i2}(X),\ldots,\lambda_{ii}(X)\}$$
be the set of all eigenvalues (with multiplicity) of $i$-th principal submatrix $X_i$. Hence, if $X,Y\in M_n(\C)$, then 

$X$ and $Y$ lie in the same fiber of $\Phi_k$ $\Longleftrightarrow$ ${\bf E}_X(i)={\bf E}_Y(i)$ for $n-k+1\le i\le n$.

\end{remark}

\begin{proposition}[Colarusso-Evens, 2015]
\label{ColEv}
\noindent

For any $\alpha\in\C^{n-1}\times\C^{n}$, the fiber $\Phi_2^{-1}\left(\alpha\right)$ of the $2$-partial Kostant-Wallach map $\Phi_2$ is equidimensional with dimension $\dim(\Phi_2^{-1}\left(\alpha)\right)=n^2-2n+1$.
\end{proposition}


\begin{proof} 
See \cite{ColEv}.

\end{proof}


\subsection{Kostant-Wallach map vs Gelfand-Tseltin varieties}

We can view the zero fiber of the Kostant-Wallach map and of its partial maps as Gelfand-Tsetlin varieties of the following form:  


\begin{definition}
For $k=1,2,\ldots,n$ and $\beta\in\C^{n-k+1}\times\cdots\times\C^{n-1}\times\C^{n}$, we define the $k$-\textbf{partial Gelfand-Tsetlin variety in} $\beta$ as the
algebraic variety
$$\widetilde{V}_{\beta}^k=V\left(\left\{\overline{\gamma}_{ij}-\beta_{ij}:\ n-k+1\leq i \leq n\ \ \mbox{and}\ \ \ 1\leq j\leq i \right\}\right)\subset\C^{n^2}.$$

\end{definition}


\begin{remark}
Clearly, $\widetilde{V}_{0}^n=V\left(\left\{\overline{\gamma}_{ij}: \ \ 1\leq j\leq i \leq n \right\}\right)$ is the Gelfand-Tsetlin variety for $\gl_n$.

\end{remark}


\begin{proposition}[\v Zelobenko, 1973]
For each $i=1,2,\dots,n$
$$V\left(\chi_{i1},\chi_{i2},\ldots,\chi_{ii}\right)=V\left(\overline{\gamma}_{i1},\overline{\gamma}_{i2},\ldots,\overline{\gamma}_{ii}\right).$$
\end{proposition}


\begin{proof}
 See \cite{Zelobenko} (page $170$) or \cite{Ovs}.
 
\end{proof}


The relation between Kostant-Wallach maps and Gelfand-Tseltin varieties is determined by the following corollary:


\begin{corollary}
\label{KWvsGTs}
For all $k=1,2,\dots,n$, the $k$-partial Gelfand-Tsetlin variety in zero coincides with the fiber in zero of the $k$-partial Kostant-Wallach map, i.e. 
$$\widetilde{V}_{0}^k=\Phi_k^{-1}(0).$$
\end{corollary}


\begin{proof} Since $\Phi_k^{-1}(0)=\left\{X\in M_n(\C)\ /\ \Phi_k(X)=0\right\}$ and
$$\begin{array}{rccl}
   \Phi_k: & M_n(\C) & \longrightarrow & \C^{n-k+1}\times\cdots\times\C^{n-1}\times\C^{n}\\
           & X       & \longmapsto     & \Phi_k(X):=\left(\chi_{n-k+1}(X),\ldots,\chi_{n-1}(X),\chi_{n}(X)\right),
  \end{array}$$
we have
\begin{align*}
\Phi_k^{-1}(0)&=\bigcap_{i=n-k+1}^n V\left(\chi_{i1},\chi_{i2},\ldots,\chi_{ii}\right)=\bigcap_{i=n-k+1}^n V\left(\overline{\gamma}_{i1},\overline{\gamma}_{i2},\ldots,\overline{\gamma}_{ii}\right)=\widetilde{V}_{0}^k.
\end{align*}

\end{proof}


\begin{remark}

This Corollary asserts that the Gelfand-Tsetlin variety coincides with the zero fiber of the Kostant-Wallach map $\Phi^{-1}(0)$. Consequently, from Remark \eqref{remarkEig} the Gelfand-Tsetlin variety for $\gl_n(\C)$ is exactly the set of the strongly nilpotent matrices $n\times n$ (where, a matrix $X\in M_n(\C)$ is said to be \textbf{strongly nilpotent}, when all its $i-$th principal submatrices $X_i$ are nilpotent).   
\end{remark}


\subsection{Equidimensionality of the fibers}

Now, we will prove a generalization of the Proposition \eqref{ColEv}.

\begin{theorem}
\label{ColEvGen}
For all $k=1,2,\ldots,n$ and all 
$\beta\in\C^{n-k+1}\times\cdots\times\C^{n-1}\times\C^{n}$, the $k$-partial Gelfand-Tsetlin variety $\widetilde{V}_{\beta}^k$ is equidimensional with dimension 
$$\dim\left(\widetilde{V}_{\beta}^k\right)=n^2-nk+\frac{k(k-1)}{2}.$$
In particular, for all $\alpha\in\C^{d(n)}$ the $n$-partial Gelfand-Tsetlin variety $\widetilde{V}_{\alpha}^n$ is equidimensional with dimension $\dim\left(\widetilde{V}_{\alpha}^n\right)=d(n-1).$
 \end{theorem}


\begin{proof}
We consider the sequence in $\gr(U(\gl_n(\C))$
$$\left\{\sigma_{ij}:=\overline{\gamma_{ij}}-\beta_{ij}\ :\ 1\leq j\leq i,\ n-k+1\leq i \leq n\right\}.$$
Note that, for $1\leq j\leq i$ and $n-k+1\leq i \leq n$, we have $\overline{\sigma_{ij}}=\overline{\overline{\gamma_{ij}}-\beta_{ij}}=\overline{\overline{\gamma_{ij}}}=\overline{\gamma_{ij}}$ in $\gr(U(\gl_n(\C))=\gr(\gr(U(\gl_n(\C)))$. Hence, the sequence 
$$\left\{\overline{\sigma_{ij}}=\overline{\gamma_{ij}}\ :\ 1\leq j\leq i,\ n-k+1\leq i \leq n\right\}$$
is a subsequence of $\left\{\overline{\gamma_{ij}}\ :\ 1\leq j\leq i\leq n\right\},$ which, by Ovsienko's Theorem \eqref{tmaOvs} and by Proposition \eqref{regvseq} is regular in $\gr(U(\gl_n(\C)))$ and from the Corollary \eqref{regsubseq} follows that $\left\{\overline{\sigma_{ij}}=\overline{\gamma_{ij}}\ :\ 1\leq j\leq i,\ n-k+1\leq i \leq n\right\}$
is a regular sequence in $\gr(U(\gl_n(\C))=\gr(\gr(U(\gl_n(\C)))$.

Consequently, by Proposition \eqref{regseq}, we get that 
$$\left\{\sigma_{ij}:=\overline{\gamma_{ij}}-\beta_{ij}\ :\ 1\leq j\leq i,\ n-k+1\leq i \leq n\right\}$$
is a regular sequence in $\gr(U(\gl_n(\C)))$. Therefore, by Propositions \eqref{defKoszul} and \eqref{regvseq}, 
$$\widetilde{V}_{\beta}^k=V\left(\left\{\overline{\gamma_{ij}}-\beta_{ij}:\ \ 1\leq j\leq i,\ n-k+1\leq i \leq n \right\}\right)\subset\C^{n^2}$$
is a equidimensional variety with dimension
$$\dim\left(\widetilde{V}_{\beta}^k\right)=n^2-n-(n-1)-(n-2)-\cdots-(n-k+1)=n^2-nk+\frac{k(k-1)}{2}.$$

\end{proof}

\vspace{0,3cm}

We can conclude the equidimensionality for the fibers of the Kostant-Wallach map and its partial maps.


\begin{corollary}
\label{genColEv}
For each $k=1,2,\ldots,n$ and each 
$\beta\in\C^{n-k+1}\times\cdots\times\C^{n-1}\times\C^{n}$, the fiber $\Phi_k^{-1}(\beta)$ of the $k$-partial Kostant-Wallach map $\Phi_k$ is equidimensional with dimension $\dim(\Phi_k^{-1}(\beta))=n^2-nk+d(k-1)$.

\end{corollary}


\begin{proof}
Note that the polynomials $\chi_{ij}$ are homogeneous in 
$\Lambda=\C[X_{ij}:1\leq i,j\leq n]$ and therefore we can now proceed analogously to the proof of the Theorem \eqref{ColEvGen}. To do this, take the sequence 
$$\left\{\chi_{ij}-\beta_{ij}:n-k+1\leq i\leq n\ \ \mbox{and}\ \ 1\leq j\leq i\right\}$$
and the graduation $\overline{\Lambda}=\Lambda$ of $\Lambda$ given by the polynomial degree (i.e. $\deg(X_{ij}^{t})=t$). 

When $n-k+1\leq i\leq n$ and $1\leq j\leq i$, we have $\overline{\chi_{ij}-\beta_{ij}}=\overline{\chi_{ij}}=\chi_{ij}$. It follows from Corollary \eqref{KWvsGTs}, that 
$$\widetilde{V}_0^{k}=\Phi_k^{-1}(0)=V\left(\left\{\chi_{ij}:n-k+1\leq i\leq n\ \ \mbox{and}\ \ 1\leq j\leq i\right\}\right),$$
which by Theorem \eqref{ColEvGen} is equidimensional with dimension $n^2-nk+d(k-1).$ Hence, by Proposition \eqref{regvseq} the sequence 
$$\left\{\overline{\chi_{ij}-\beta_{ij}}=\chi_{ij}:n-k+1\leq i\leq n\ \ \mbox{and}\ \ 1\leq j\leq i\right\}$$
is regular in $\overline{\Lambda}$ and by Proposition \eqref{regseq} we can conclude that the sequence
$$\left\{\chi_{ij}-\beta_{ij}:n-k+1\leq i\leq n\ \ \mbox{and}\ \ 1\leq j\leq i\right\}$$
is regular in $\Lambda$. Finally, it follows from Propositions \eqref{defKoszul}, \eqref{regvseq} and the definition of complete intersections, that the variety 
$$\Phi_k^{-1}(\beta)=V\left(\left\{\chi_{ij}-\beta_{ij}:n-k+1\leq i\leq n\ \ \mbox{and}\ \ 1\leq j\leq i\right\}\right)$$
is equidimensional with dimension $\dim(\Phi_k^{-1}(\beta))=n^2-nk+d(k-1).$

\end{proof}


\begin{corollary}

For each $\alpha\in\C^{d(n)}$, the fiber $\Phi^{-1}(\alpha)$ of the Kostant-Wallach map $\Phi$ is equidimensional with dimension $\dim(\Phi^{-1}(\alpha))=d(n-1).$

\end{corollary}


\begin{proof}
It follows from the fact that $\Phi_n=\Phi$ and from previous corollary.

\end{proof}


\section{Some cases of Ovsienko's Theorem}

In this section we will give another proof of Ovsienko's Theorem for $\gl_2$, $\gl_3$ and $\gl_4$ using different techniques used by Ovsienko in \cite{Ovs}. Denote the Gelfand-Tsetlin variety for $\gl_n$ by $\gts_n$.


\begin{proposition}
\label{OvsThgl4}
The Gelfand-Tsetlin variety for $\gl_4$ is equidimensional of dimension $6$.
\end{proposition}

\begin{proof}
Consider the following decomposition 
$$\gts_4=V_4\cup A\cup B\cup C\cup D\subset\k^{16},$$ where
 $$V_4=V\left(\overline{\gamma}_{11},\overline{\gamma}_{21},X_{12},\overline{\gamma}_{31},X_{23},X_{13},\overline{\gamma}_{41},\overline{\gamma}_{42},\overline{\gamma}_{43},\overline{\gamma}_{44}\right),$$
$$A=V\left(\overline{\gamma}_{11},\overline{\gamma}_{21},X_{12},\overline{\gamma}_{31},X_{32},X_{13},\overline{\gamma}_{41},\overline{\gamma}_{42},\overline{\gamma}_{43},\overline{\gamma}_{44}\right),$$
$$B=V\left(\overline{\gamma}_{11},\overline{\gamma}_{21},X_{12},\overline{\gamma}_{31},\overline{\gamma}_{32},X_{21},\overline{\gamma}_{41},\overline{\gamma}_{42},\overline{\gamma}_{43},\overline{\gamma}_{44}\right),$$
$$C=V\left(\overline{\gamma}_{11},\overline{\gamma}_{21},X_{12},\overline{\gamma}_{31},\overline{\gamma}_{32},X_{32},\overline{\gamma}_{41},\overline{\gamma}_{42},\overline{\gamma}_{43},\overline{\gamma}_{44}\right),$$
$$D=V\left(\overline{\gamma}_{11},\overline{\gamma}_{21},X_{21},\overline{\gamma}_{31},\overline{\gamma}_{32},\overline{\gamma}_{33},\overline{\gamma}_{41},\overline{\gamma}_{42},\overline{\gamma}_{43},\overline{\gamma}_{44}\right).$$
Also, the variables matrix $X=\left(X_{ij}\right)_{i,j=1}^n$ and using the variable change\footnote{$X^t$ is the transpose matrix of $X$}:
\begin{enumerate}
\item $X\longmapsto X^t$, we have 
$$V_4\cup A\cup B\cup C=V\left(\overline{\gamma}_{11},\overline{\gamma}_{21},X_{12},\overline{\gamma}_{31},\overline{\gamma}_{32},\overline{\gamma}_{33},\overline{\gamma}_{41},\overline{\gamma}_{42},\overline{\gamma}_{43},\overline{\gamma}_{44}\right)\cong D.$$
\item $X\longmapsto (E_{12}+E_{21}+E_{33}+E_{44})X^t(E_{12}+E_{21}+E_{33}+E_{44})$ it follows that 
$$V_4\cup A=V\left(\overline{\gamma}_{11},\overline{\gamma}_{21},X_{12},\overline{\gamma}_{31},\overline{\gamma}_{32},X_{13},\overline{\gamma}_{41},\overline{\gamma}_{42},\overline{\gamma}_{43},\overline{\gamma}_{44}\right)\cong C.$$
\item $X\longmapsto (E_{11}+E_{23}+E_{32}+E_{44})X(E_{11}+E_{23}+E_{32}+E_{44})$ we obtain that 
$$V_4\cong A.$$
\end{enumerate}

Due $V_4$ is the weak version of the Gelfand-Tsetlin variety for $\gl_4$, from Theorem \eqref{fracagln} we only need to prove that $B$ is equidimensional of $\dim B=6$, which by Corollary \eqref{regsubseq} and Proposition \eqref{regvseq} it is enough to prove that the variety $B_1=B\cap V\left(X_{13}-X_{14},X_{31}-X_{41}\right)\subset\k^{16}$ is equidimensional of dimension $4$. Observe that $B_1=\overline{B}_1\cup\widehat{B}_1$, where
$$\Omega=\{X_{11},X_{22},X_{12},X_{33},X_{21},X_{44}\},$$
$$\overline{B}_1=V\left(\{X_{13}-X_{14},X_{31}-X_{41}\}\cup\Omega\cup\{\overline{\gamma}_{32},\overline{\gamma}_{42},\overline{\gamma}_{43},X_{31}X_{13}\}\right),$$
$$\widehat{B}_1=V\left(\{X_{13}-X_{14},X_{31}-X_{41}\}\cup\Omega\cup\{\overline{\gamma}_{32},\overline{\gamma}_{42},\overline{\gamma}_{43},X_{32}X_{24}+X_{42}X_{23}\}\right).$$
Clearly, $\overline{B}_1$ is equidimensional of $\dim\overline{B}_1=4$, repeating the previous argument over $\widehat{B}_1$, combining Corollary \eqref{regsubseq} and Proposition \eqref{regvseq} it is sufficient to show that the variety $B_2=\widehat{B}_1\cap V\left(X_{43}-X_{34}\right)$ is equidimensional of $\dim B_2=3$. Note $B_2=\overline{B}_2\cup\widehat{B}_2$, where respectively $\overline{B}_2$ and $\widehat{B}_2$ are
$$V\left(X_{43}-X_{34},X_{13}-X_{14},X_{31}-X_{41},\overline{\gamma}_{32},\overline{\gamma}_{42},X_{31}X_{13},X_{32}X_{24}+X_{42}X_{23}\right)\cap V(\Omega),$$
$$V\left(\{X_{43}-X_{34},X_{13}-X_{14},X_{31}-X_{41}\}\cup\Omega\cup\{\overline{\gamma}_{32},\overline{\gamma}_{42},X_{34},X_{32}X_{24}+X_{42}X_{23}\}\right).$$
Similarly, it is easy to check that $\overline{B}_2$ is equidimensional of $\dim\overline{B}_2=3$ and for
$$\widehat{B}_2=V\left(\{X_{43},X_{13}-X_{14},X_{31}-X_{41}\}\cup\Omega\cup\{\overline{\gamma}_{32},\overline{\gamma}_{42},X_{34},X_{32}X_{24}+X_{42}X_{23}\}\right),$$
from Corollary \eqref{regsubseq} and Proposition \eqref{regvseq} it is enough to prove that the variety $B_3=\widehat{B}_2\cap V\left(X_{23}-X_{32}\right)$ is equidimensional of $\dim B_3=2$. Since $B_3=\overline{B}_3\cup\widehat{B}_3$ with
$$\overline{B}_3=V\left(\{X_{23}-X_{32},X_{43},X_{13}-X_{14},X_{31}-X_{41}\}\cup\Omega\cup\{\overline{\gamma}_{32},\overline{\gamma}_{42},X_{34},X_{23}\}\right),$$
$$\widehat{B}_3=V\left(\{X_{23}-X_{32},X_{43},X_{13}-X_{14},X_{31}-X_{41}\}\cup\Omega\cup\{\overline{\gamma}_{32},\overline{\gamma}_{42},X_{34},X_{24}+X_{42}\}\right).$$
Also, it is not hard to see that $\overline{B}_3$ is equidimensional of $\dim\overline{B}_3=2$ and over $\widehat{B}_3$, by Corollary \eqref{regsubseq} and Proposition \eqref{regvseq} the proof is completed by showing that $B_4=\widehat{B}_3\cap V\left(X_{23}\right)$ is equidimensional of dimension $1$. Finally, $B_4=\overline{B}_4\cup\widehat{B}_4$, where
$$\overline{B}_4=V\left(\{X_{23},X_{32},X_{43},X_{13}-X_{14},X_{41}\}\cup\Omega\cup\{X_{31},X_{24},X_{34},X_{42}\}\right),$$
$$\widehat{B}_4=V\left(\{X_{23},X_{32},X_{43},X_{14},X_{31}-X_{41}\}\cup\Omega\cup\{X_{13},X_{24},X_{34},X_{42}\}\right),$$
which both are equidimensional of dimension $1$.

\end{proof}

 
\begin{theorem}
For $n=2,3,4$, the Gelfand-Tsetlin variety for $\gl_n$ is equidimensional of dimension $\dfrac{n(n-1)}{2}$.
\end{theorem}

\begin{proof} Clearly, $\gts_2=V(X_{11},X_{22},X_{21})\cup V(X_{11},X_{22},X_{12})\subset\k^4$ is equidimensional of $\dim\gts_2=1.$ Also, $\gts_3=V_3\cup W\subset \k^{9}$, where 
$$V_3=V\left(\overline{\gamma}_{11},\overline{\gamma}_{21},X_{12},\overline{\gamma}_{31},\overline{\gamma}_{32},\overline{\gamma}_{33}\right)\ \ \ \mbox{and}\ \ \ W=V\left(\overline{\gamma}_{11},\overline{\gamma}_{21},X_{21},\overline{\gamma}_{31},\overline{\gamma}_{32},\overline{\gamma}_{33}\right).$$
Using the variable change $X_{ij}\longmapsto X_{ji}$, we have $V_3\cong W$ and due $V_3$ is the weak version of the Gelfand-Tsetlin variety, by Theorem \eqref{fracagln} $\gts_3$ is equidimensional of $\dim\gts_3=3$.

Finally, from Proposition \eqref{OvsThgl4} $\gts_4$ is equidimensional of dimension $\dim\gts_4=6$.

\end{proof}


\end{document}